\documentclass[12pt]{amsart}
\usepackage{amsmath,amssymb}
\usepackage{amsmath}
\usepackage{amsthm}
\usepackage{amsfonts}
\usepackage{graphicx}
\usepackage{float}
\usepackage{enumerate}
\usepackage{extarrows}
\usepackage{amssymb}
\usepackage{mathtools}
\usepackage{mathrsfs}
\usepackage{amsmath,amscd}
\usepackage{xy}
\xyoption{all}

\theoremstyle{plain}
\newtheorem{thm}{Theorem}[section]
\newtheorem{theorem}[thm]{Theorem}

\newtheorem{lemma}[thm]{Lemma}
\newtheorem{corollary}[thm]{Corollary}
\newtheorem{proposition}[thm]{Proposition}
\theoremstyle{definition}
\newtheorem{remark}[thm]{Remark}

\newtheorem{definition}[thm]{Definition}

\numberwithin{equation}{section}
\newcommand{\sD}{{\mathcal D}}
\newcommand{\sE}{{\mathcal E}}

\newcommand{\R}{{\mathbb R}}

\title[A remark on Euler-like vector fields]{A remark on Euler-like vector fields}
\author{Haoyuan Gao}
\begin{document}
\maketitle  \setcounter{tocdepth}{1}

\begin{abstract}
In this note, we show that (the germ of) each Euler-like vector field comes from a tubular neighborhood embedding given by the normal exponential map of some Riemannian metric.
\end{abstract}

\section{Introduction}

The notion of Euler-like vector field was introduced by Bursztyn, Lima and Meinrenken to deal with splitting theorems for some geometric structures, e.g., Poisson structures, Lie algebroids, Dirac structures and generalized complex structures \cite{BLM}. Euler-like vector fields also enable us to obtain simpler proofs of some classical results, e.g., the Darboux theorem and the Morse-Bott lemma \cite{HSSH, M}. In \cite{HSSH}, an algebraic characterization of Euler-like vector fields was given. Based on their algebraic point of view, Haj Saeedi Sadegh and Higson developed the method of deformation to the normal cone from algebraic geometry to differential geometry, and obtained a new approach to Connes' tangent groupoid. See also \cite{H} for deformation to the normal cone in differential geometry. We refer to \cite{HY, BBLM, Mo} for further developments and applications.

We recall the notion of Euler-like vector field and its alternative characterization by tubular neighborhood embedding as follows. First recall that the Euler vector field on a real vector space $V$ of dimension $k$ with linear coordinates $(x^1, \cdots, x^k)$ is defined to be
$$\sE = \sum\limits_{i = 1}^k x^i \frac{\partial}{\partial x^i}.$$
It is clear that the definition of $\sE$ dose not depend on the choice of linear coordinates. The notion of Euler vector field can be directly generalized to vector bundles. Let $E \to M$ be a smooth vector bundle of rank $k$ over a smooth manifold $M$. The Euler vector field on $E$ is defined to be
$$\sE = \sum\limits_{i = 1}^k x^i \frac{\partial}{\partial x^i},$$
where $(x^1, \cdots, x^k)$ are linear coordinates in the fiber direction. It is clear that $\sE$ is globally defined. To introduce the notion of Euler-like vector field, we need to recall the notion of linear approximation of a vector field. Let $M$ be a smooth manifold, and $N \subset M$ an embedded smooth submanifold. Denote $\nu(M, N) = TM|_N / TN$ the normal bundle of $N$ in $M$. We write $\nu_N = \nu(M, N)$ if the ambient manifold is clear. Let $X \in \Gamma(TM)$ be a smooth vector field tangent to N. Then we have a smooth map of pairs $X : (M, N) \to (TM, TN)$. Then the differential of $X$ induces a bundle map
$$\nu(X) : \nu(M, N) \to \nu(TM, TN).$$
Note that there is a canonical isomorphism $\nu(TM, TN) \cong T\nu(M, N)$, c.f. \cite{BLM}. In this sense, we obtain a smooth vector field
$$\nu(X) : \nu_N \to T\nu_N$$
on $\nu_N$. The vector field $\nu(X) \in \Gamma(T\nu_N)$ is called the linear approximation of $X$. The Euler vector field $\sE$ on a vector bundle $E \to M$ is zero on the zero-section $M \subset E$. Under the canonical isomorphism $\nu(E, M) \cong E$, c.f., \cite{BLM}, it is easy to see that the linear approximation of $\sE$ is itself. Now we introduce Euler-like vector field and (a strong version of) tubular neighborhood embedding.
\begin{definition} \cite[Definition 2.6]{BLM}
    Let $M$ be a smooth manifold, and $N \subset M$ an embedded smooth submanifold. A smooth vector field $X \in \Gamma(TM)$ is called Euler-like (along $N$) if $X|_N = 0$, and its linear approximation is the Euler vector field, i.e., $\nu(X) = \sE$.
\end{definition}
\begin{definition} \cite[Definition 2.3]{BLM} \label{def-tubular-nbd-embedding}
    Let $M$ be a smooth manifold, and $N \subset M$ an embedded smooth submanifold. A tubular neighborhood embedding for $N \subset M$ is an embedding
    $$\psi : \nu_N \to M,$$
    taking the zero-section of $\nu_N$ to $N$, and such that the map $\nu(\psi)$ induced by
    $$\psi: (\nu_N, N) \to (M, N)$$
    is the identity map of $\nu_N$ under the canonical isomorphism $\nu(\nu_N, N) \cong \nu_N = \nu(M, N)$.
\end{definition}
\begin{remark}
    Let $\psi : \nu(M, N) \to M$ be a smooth map taking the zero-section of $\nu_N = \nu(M, N)$ to $N$. Using local smooth coordinates of $M$ in which $N$ is given by a slice, one can see that under the canonical isomorphism $\nu(\nu_N, N) \cong \nu_N = \nu(M, N)$, the induced map
    $$\nu(\psi) : \nu(M, N) \to \nu(M, N)$$
    is given by
    $$\nu(\psi)(p, v) = \big(\psi(p), d\psi_{(p, 0)}(v)\big), \quad \forall p \in N, \,\, \forall v \in \nu(M, N)|_p = T_pM/T_pN.$$
    Here we identified $N$ with the zero-section of $\nu_N$, and used the canonical isomorphism $T_{(p, 0)}\nu_N = T_pN \oplus \nu(M, N)|_p$ in $d\psi_{(p, 0)}(v)$, and $d\psi_{(p, 0)}$ denotes the linear map
    $$d\psi_{(p, 0)} : T_{(p, 0)}\nu_N / T_{(p, 0)}N \to T_{\psi(p)}M / T_{\psi(p)}N$$
    induced by the tangent map
    $$d\psi_{(p, 0)} : T_{(p, 0)}\nu_N \to T_{\psi(p)}M.$$
    Therefore, the induced map $\nu(\psi)$ depends only on $\psi$ near $N$. If $\psi : \nu_N \to M$ is a tubular neighborhood embedding, then we have that the restriction of $\psi$ to $N$ is the identity map of $N$, and the differential of $\psi$, restricted to the vertical tangent vectors, also induces identity map from $\nu_N$ to itself.
\end{remark}

Given a tubular neighborhood embedding $\psi : \nu_N \to M$, it is easy to see that the pushforward vector field $\psi_* \sE$ is Euler-like along $N$ in $\psi(\nu_N)$. Here $\sE$ is the Euler vector field on the normal bundle $\nu_N$, and $\psi_* \sE$ is defined by
$$(\psi_* \sE)_{\psi(v)} = d\psi_{v}(\sE_v), \quad \forall v \in \nu_N.$$
Conversely, given an Euler-like vector field $X$ on $M$, it was shown in \cite{BLM} that the flow of $X$ provides a unique (germ of) tubular neighborhood embedding $\psi$ such that $X = \psi_*(\sE)$ in an open neighborhood of $N$. In fact, Bursztyn, Lima and Meinrenken proved that the above correspondence between tubular neighborhood embeddings and Euler-like vector fields is one-to-one (in the level of germ). We have the following theorem.
\begin{theorem} \cite[Proposition 2.7]{BLM} \label{one-to-one-tubular-Euler-like}
    The correspondence that associates to a tubular neighborhood embedding its associated (pushforward) Euler-like vector field determines a bijection from germs of tubular neighborhood embeddings to germs of Euler-like vector fields. Here a germ means a germ near $N$.
\end{theorem}

\begin{remark}
    Here we use the statement from \cite{HSSH} instead of the original statement in \cite{BLM} because we do not require an Euler-like vector field to be complete. See also Remark 2.9 in \cite{BLM}.
\end{remark}

In the following of this note, we show that (the germ of) each tubular neighborhood embedding can be realized by the normal exponential map of some Riemannian metric on (an open subset of) the ambient manifold. In Section 2, we review basic properties of the normal exponential map, and give a precise statement of the main theorem. In Section 3, we prove the main theorem. In the appendix, we consider extension of smooth functions on a vector bundle. This enables us to discuss germs (of tubular neighborhood embeddings) freely.

Throughout this note, all smooth manifolds are assumed to be Hausdorff and second countable. When we say a Riemannian metric, we mean a Riemannian metric on a smooth manifold, i.e., a smooth fiberwise metric on the tangent bundle of the given manifold. For a vector bundle, to distinguish from Riemannian metrics, we use the terminology of bundle metrics to stand for fiberwise metrics on the given vector bundle.

\vspace{5mm}

{\em Acknowledgements:} The author would like to thank Zelin Yi for useful discussions.

\section{Normal exponential map}

In this section, we review basic properties of the normal exponential map of a Riemannian metric along an embedded smooth submanifold.

Let $(M, g)$ be a smooth Riemannian manifold, and $p \in M$ a point in $M$. For $v \in T_pM$, we denote $\gamma_v$ the unique maximal geodesic with initial point $p$ and initial velocity vector $v$, i.e., $\gamma_v(0) = p$, and $\gamma_v'(0) = v$. Denote
$$\sD_p := \big\{v \in T_pM \, \big| \, {\rm The \,\, geodesic \,\,}\gamma_v \,\, {\rm is \,\, defined \,\, on \,\, an \,\, interval \,\, containing \,\,}[0, 1]\big\}.$$
Recall that the exponential map at $p$
$$\exp_p : \sD_p \to M$$
is defined by
$$\exp_p(v) = \gamma_v(1), \quad \forall v \in \sD_p.$$
The rescaling property of geodesics indicates that $\sD_p$ is nonempty. Let $\sD = \bigcup\limits_{p \in M}\sD_p \subset TM$. The collection of $\exp_p$ yields a map
$$\exp : \sD \to M, \quad \exp(v) = \exp_{\pi(p)}(v), \quad \forall v \in TM,$$
called the exponential map of $g$, where $\pi : TM \to M$ is the projection. The exponential map has the following properties.

\begin{proposition} \cite[Proposition 5.19]{L} \label{property-exp-map}
    Let $(M, g)$ be a Riemannian manifold, and let $\exp : \sD \to M$ be the exponential map defined as above.
    \begin{itemize}
        \item [(i)] $\sD$ is an open subset of $TM$ containing the zero-section, and each set $\sD_p \subset T_pM$ is star-shaped with respect to $0 \in T_pM$.
        \item [(ii)] For each $v \in TM$, the maximal geodesic $\gamma_v$ is given by
        $$\gamma_v(t) = \exp(tv)$$
        for all $t$ such that either side is defined.
        \item [(iii)] The exponential map is smooth.
        \item [(iv)] For each $p \in M$, the differential $d(\exp_p)_0 : T_0(T_pM) \cong T_pM \to T_pM$ is the identity map of $T_pM$, under the usual identification of $T_0(T_pM)$ with $T_pM$.
    \end{itemize}
\end{proposition}

Now let $N \subset M$ be an embedded smooth submanifold. We have the Riemannian normal bundle $\Lambda N \to N$ of $N$ in $M$ defined by
$$\Lambda N = \bigsqcup\limits_{p \in N}\Lambda_p N,$$
where
$$\Lambda_p N = \big\{v \in T_pM \, \big| \, g(v, w) = 0, \,\, \forall w \in T_pN\big\}.$$
It is clear that $\Lambda N$ is a subbundle of $TM|_N$, and is canonically isomorphic to the normal bundle $\nu_N = \nu(M, N)$.

Let $\sD_N = \sD \cap \Lambda N$, and $E : \sD_N \to M$ the restriction of the exponential map. The map $E$ is called the normal exponential map of $N$ in $M$. A normal neighborhood of $N$ in $M$ is an open subset $U \subset M$ that is the diffeomorphic image under $E$ of an open subset $V \subset \sD_N$ whose intersection with each fiber $\Lambda_pN$ is star-shaped with respect to $0$. A normal neighborhood of $N$ in $M$ is called a normal tubular neighborhood if it is the diffeomorphic image under $E$ of a subset $V \subset \sD_N$ of the form
\begin{equation} \label{def-normal-tubular}
    V = \{(p, v) \in \Lambda N \, \big| \, |v|_g < \delta(p)\},
\end{equation}
for some positive continuous function $\delta : N \to \R$. The existence of normal tubular neighborhood was proved in \cite{L}.

\begin{proposition} \cite[Theorem 5.25]{L} \label{existence-normal-tubular}
    Let $(M, g)$ be a Riemannian manifold. Every embedded smooth submanifold of $M$ has a normal tubular neighborhood in $M$.
\end{proposition}

\begin{remark}
    Given a positive continuous function $\delta : M \to \R$ on a smooth manifold $M$, using a partition of unity subordinate to a suitable cover, it is easy to see that there is a positive smooth function $\tilde{\delta} : M \to \R$ such that $\tilde{\delta}(p) < \delta(p)$, $\forall p \in M$. Hence if we require that the function $\delta$ in (\ref{def-normal-tubular}) to be positive and smooth, Proposition \ref{existence-normal-tubular} still holds.
\end{remark}

Under the canonical isomorphism $\Lambda N \cong \nu_N$, a normal tubular neighborhood of $N$ in $M$ can be viewed as an embedding
$$\psi : V \to M,$$
where $V \subset \nu_N$ is an open subset containing the zero-section $N$. Then by Proposition \ref{property-exp-map}, the induced map
$$\nu(\psi) : \nu(V, N) = \nu(M, N) \to \nu(M, N)$$
is the identity map of $\nu(M, N)$. Therefore, any embedding $\tilde{\psi} : \nu_N \to M$ that agrees with $\psi$ near $N$ is a tubular neighborhood embedding in the sense of Definition \ref{def-tubular-nbd-embedding}. In fact, Corollary \ref{cor-extension-vector-bundle} ensures that such an embedding $\tilde{\psi}$ that agrees with $\psi$ in a neighborhood of $N$ always exists. Therefore, tubular neighborhood embedding in the sense of Definition \ref{def-tubular-nbd-embedding} always exists. In fact, we shall prove that each tubular neighborhood embedding can be realized by the normal exponential map of some Riemannian metric near $N$. More precisely, we have the following main theorem.

\begin{theorem} \label{main-thm}
    Let $M$ be a smooth manifold, and $N \subset M$ an embedded smooth submanifold. Let $\psi : \nu_N \to M$ be a tubular neighborhood embedding. Then we have an open subset $U \subset \psi(\nu_N)$ containing $N$ and a Riemannian metric $g$ on $U$ such that $U$ is the diffeomorphic image under the normal exponential map of an open subset $\widehat{U} \subset \Lambda N$ of the Riemannian normal bundle $\Lambda N$ in $U$ containing the zero-section. Moreover, we have the following commutative diagram
    \begin{equation} \label{main-diagram}
        \xymatrix{
        \Psi^{-1}(U)
        \ar[dr]_{\psi} \ar[r]^{\Psi} & \widehat{U} \ar[d]^{E}\\
        & U,
        }
    \end{equation}
    where $\Psi : \nu(M, N) \to \Lambda N$ is the canonical bundle isomorphism, and the vertical map $E$ is the normal exponential map.
\end{theorem}

\section{Proof of the main theorem}

In this section, we prove Theorem \ref{main-thm} which claims that each tubular neighborhood embedding can be realized by the normal exponential map of some Riemannian metric near the given submanifold.

Before we deal with general cases, we first consider the simplest (non-trivial) case in which the submanifold $N$ is a single point.

\begin{proposition}
    Let $M$ be a smooth manifold, and $p \in M$ a point in $M$. Let $\psi : T_pM \to M$ be an embedding such that $\psi(0) = p$ and $d\psi_0 : T_0(T_pM) \cong T_pM \to T_pM$ is the identity map of $T_pM$ under the canonical identification of $T_0(T_pM)$ with $T_pM$. Then there is a Riemannian metric $g$ on $\psi(T_pM)$ such that the exponential map of $g$ at the point $p \in M$ is defined on the whole tangent space $T_pM$
    $$\exp_p : T_pM \to M,$$
    and
    $$\psi = \exp_p.$$
\end{proposition}

\begin{proof}
    Since $T_pM$ is a vector space, we may choose a Riemannian metric $\tilde{g}$ on $T_pM$ whose Levi-Civita connection has vanishing Christoffel symbols with respect to all linear coordinate systems. Then define a Riemannian metric $g$ on $\psi(T_pM)$ by $g := \psi_*(\tilde{g}) = (\psi^{-1})^*(\tilde{g})$. For $v \in T_pM$, we have that $(d\psi_0)^{-1}(v) = v \in T_0(T_pM) \cong T_pM$ by assumption. Since the Levi-Civita connection of $\tilde{g}$ has vanishing Christoffel symbols with respect to all linear coordinate systems, the maximal geodesic $\tilde{\gamma}_v$ in $(T_pM, \tilde{g})$ with $\tilde{\gamma}_v(0) = 0$ and $\tilde{\gamma}'_v(0) = v$ is given by
    $$\tilde{\gamma}_v : \R \to T_pM, \quad \tilde{\gamma}_v(t) = tv, \,\, \forall t \in \R.$$
    Therefore, the maximal geodesic $\gamma_v$ in $\psi(T_pM)$ with $\gamma_v(0) = p$ and $\gamma_v'(0) = v$ is given by
    $$\gamma_v : \R \to M, \quad \gamma_v(t) = \psi\big(\tilde{\gamma}_v(t)\big) = \psi(tv), \,\, \forall t \in \R.$$
    Hence we have
    $$\exp_p(v) = \gamma_v(1) = \psi(v)$$
    for all $v \in T_pM$.
\end{proof}

Now we move on to general cases. Note that the key point in the proof of the case for a single point is to pushforward a suitable metric so that we can write geodesics in the target open subset explicitly.

\begin{proof}[Proof of Theorem \ref{main-thm}]
    Let $\tilde{g}$ be a Riemannian metric on $M$. Denote $\tilde{\Lambda}N = \bigsqcup\limits_{p \in N}\tilde{\Lambda}_p N$ the Riemannian normal bundle of $N$ in $(M, \tilde{g})$. Let $\widetilde{\exp}$ be the exponential map of $\tilde{g}$ and $\tilde{E}$ the normal exponential map of $N$ in $(M, \tilde{g})$. Denote $\Phi : \nu(M, N) \to \tilde{\Lambda}N$ be the canonical bundle isomorphism. Then by Proposition \ref{existence-normal-tubular}, we have an open subset $\widehat{V} \subset \tilde{\Lambda}N$ containing the zero-section such that $V = \tilde{E}(\widehat{V})$ is open in $M$, and $\tilde{E} : \widehat{V} \to V$ is a diffeomorphism. Then we have an embedding
    $$\varphi : \Phi^{-1}(\widehat{V}) \to M$$
    defined by $\varphi = \tilde{E} \circ \Phi$. Then we have the following commutative diagram of diffeomorphisms
    \begin{equation*}
        \xymatrix{
        \Phi^{-1}(\widehat{V})
        \ar[dr]_{\varphi} \ar[r]^{\Phi} & \widehat{V} \ar[d]^{\tilde{E}}\\
        & V.
        }
    \end{equation*}
    Moreover, by Proposition \ref{property-exp-map} we have that the induced map
    $$\nu(\varphi) : \nu(M, N) \to \nu(M, N)$$
    is the identity map of $\nu_N = \nu(M, N)$.

    Now let $\psi : \nu_N \to M$ be a tubular neighborhood embedding as in the assumption of Theorem \ref{main-thm}. Let $U = \psi\big(\Phi^{-1}(\widehat{V})\big) \subset M$. Then $U$ is an open subset of $M$, and we have a diffeomorphism
    $$\chi : U \to V$$
    given by $\chi = \varphi \circ \psi^{-1}$. It is clear that the induced map
    $$\nu(\chi) : \nu(U, N) \to \nu(V, N)$$
    is the identity map of $\nu_N = \nu(M, N) = \nu(U, N) = \nu(V, N)$. In particular, $\chi|_N$ is the identity map of $N$. Let $g$ be the Riemannian metric on $U$ given by $g = \chi^*\tilde{g}$ with the exponential map $\exp$. Denote $\Lambda N = \bigsqcup\limits_{p \in N}\Lambda_p N$ the Riemannian normal bundle of $N$ in $(U, g)$ with the normal exponential map $E$ of $N$ in $(U, g).$ Since $\chi : (U, g) \to (V, \tilde{g})$ is an isometry and $\chi|_N = {\rm Id}_N$, the differential of $\chi$ restricts to a bundle isomorphism
    $$d\chi : \Lambda N \to \tilde{\Lambda}N.$$
    Moreover, for $p \in N$, since $\nu(\chi) : \nu_N \to \nu_N$ is the identity map of $\nu_N$, we have a linear map
    $$\eta_p : \Lambda_p N \to T_pN$$
    such that
    \begin{equation} \label{d-chi}
        d\chi_p(v) = v + \eta_p(v), \quad \forall v \in \Lambda_pN.
    \end{equation}
    Let $\widehat{U} = (d\chi)^{-1}(\widehat{V})$. Then $\widehat{U}$ is an open subset of $\Lambda N$ containing the zero-section. For $v \in \widehat{U} \cap T_pM$ with $p \in N$, since $\chi : (U, g) \to (V, \tilde{g})$ is an isometry, we have
    \begin{equation} \label{tilde-exp-chi-exp}
        \widetilde{\exp}_p\Big(t\big(v + \eta_p(v)\big)\Big) = \chi\big(\exp_p(tv)\big)
    \end{equation}
    for every $t$ whenever $\widetilde{\exp}_p\Big(t\big(v + \eta_p(v)\big)\Big)$ is defined and in $V$. Hence (\ref{tilde-exp-chi-exp}) holds for $t$ in an interval containing $[0, 1]$. In particular, we have
    $$\chi\big(E(v)\big) = \chi\big(\exp_p(v)\big) = \widetilde{\exp}_p\big(v + \eta_p(v)\big) = \tilde{E}\big(v + \eta_p(v)\big).$$
    Then since $\tilde{E}|_{\widehat{V}} = \varphi \circ (\Phi|_{\widehat{V}})^{-1}$, we have
    $$\chi\big(E(v)\big) = \varphi\Big(\Phi^{-1}\big(v + \eta_p(v)\big)\Big) = \varphi(v + T_pN).$$
    Since $\chi = \varphi \circ \psi^{-1}$, we have
    $$E(v) = \psi(v + T_pN).$$
    Since $v \in \widehat{U}$ is arbitrary, we have the following commutative diagram
    \begin{equation*}
        \xymatrix{
        \Psi^{-1}(\widehat{U})
        \ar[dr]_{\psi} \ar[r]^{\Psi} & \widehat{U} \ar[d]^{E}\\
        & U,
        }
    \end{equation*}
    where $\Psi : \nu(M, N) \to \Lambda N$ is the canonical bundle isomorphism. By (\ref{d-chi}), we have that $\Psi^{-1}(\widehat{U}) = \Phi^{-1}(\widehat{V})$. Then since $U = \psi\big(\Phi^{-1}(\widehat{V})\big)$, the maps $\psi$ and $\Psi$ in the above diagram are both diffeomorphisms. Hence the map $E$ in the above diagram is also a diffeomorphism. Therefore, the theorem has been proved.
\end{proof}

\begin{remark}
    In the proof of Theorem \ref{main-thm}, we see that if we choose an initial Riemannian metric $\tilde{g}$ on the ambient manifold $M$, then (the germ of) each tubular neighborhood embedding can be realized by the normal exponential map of a Riemannian metric $g$ near $N$ such that $g$ is isometric to $\tilde{g}$ near $N$.
\end{remark}

By Theorem \ref{one-to-one-tubular-Euler-like} and Theorem \ref{main-thm}, we have that (the germ of) each Euler-like vector field on $M$ along $N$ comes from a Riemannian metric near $N$.

\appendix
\section{}
In the appendix, we consider extension of maps on a vector bundle. Let $\pi : E \to M$ be a smooth vector bundle over a smooth manifold $M$. Let $V \subset E$ be an open subset containing the zero-section $M$. We shall prove that there are open subsets $W' \subset W \subset V$ of $V$ containing the zero-section $M$, and a diffeomorphism
$$\Phi : W \to E$$
such that $\Phi|_{W'} = {\rm Id}_{W'}$. Then given a smooth map $F : V \to M'$ from $V$ to a smooth manifold $M'$, we have a smooth map $\tilde{F} : E \to M'$ such that $\tilde{F}|_{W'} = F|_{W'}$.

Let $g$ be a smooth bundle metric on $E$. Using a partition of unity subordinate to a suitable cover of $M$, it is easy to see that there is a positive smooth function $\delta : M \to \R$ such that
$$\big\{(x, v) \in E \big| |v|_g < \delta(x)\big\} \subset V.$$
Let $W := \big\{(x, v) \in E \big| |v|_g < \delta(x)\big\}$, and $W':= \big\{(x, v) \in E \big| |v|_g < \frac{1}{2}\delta(x)\big\}$. To obtain a diffeomorphism
$$\Phi : W \to E$$
that restricts to the identity map of $W'$, we need the following lemma.

\begin{lemma} \label{lem-diffeo-interval-R}
    There exists a diffeomorphism $\sigma : (-1, 1) \to \mathbb{R}$ of the form
    $$\sigma(t) = \eta(t)t, \quad \forall t \in (-1, 1),$$
    where $\eta : (-1, 1) \to \mathbb{R}$ is a smooth positive function such that $\eta(t) = \eta(|t|)$, $\forall t \in (-1, 1)$, and $\eta|_{[-\frac{1}{2}, \frac{1}{2}]} \equiv 1$.
\end{lemma}
\begin{proof}
    Let $\varphi : (-1, 1) \to \mathbb{R}$ be the diffeomorphism obtained by combining the diffeomorphism form $(-1, 1)$ to the half unit circle and the stereographic projection, i.e.,
    $$\varphi(t) = \frac{t}{\sqrt{1 - t^2}}, \quad \forall t \in (-1, 1).$$
    Let $\rho : (-1, 1) \to \mathbb{R}$ be a smooth function such that $\rho|_{[-\frac{1}{2}, \frac{1}{2}]} \equiv 0$, $\rho|_{(-1, - \frac{3}{4}] \cup [\frac{3}{4}, 1)} \equiv 1$, $\rho$ is strictly increasing in $[\frac{1}{2}, \frac{3}{4}]$, and $\rho(t) = \rho(|t|)$, $\forall t \in (-1, 1)$. Then define $\sigma : (-1, 1) \to \mathbb{R}$ by
    $$\sigma(t) = \rho(t)\varphi(t) + t = \Big(\frac{\rho(t)}{\sqrt{1 - t^2}} + 1\Big)t, \quad \forall t \in (-1, 1).$$
    It is clear that $t \mapsto \rho(t)\varphi(t)$ is non-decreasing. Hence $\sigma'(t) \geq 1$, $\forall t \in (-1, 1)$. Then since $\sigma\big((-1, 1)\big) = \mathbb{R}$ and $\sigma$ is increasing, we have that $\sigma$ is a diffeomorphism. Let $\eta(t) = \frac{\rho(t)}{\sqrt{1 - t^2}} + 1$, $\forall t \in (-1, 1)$. It is clear that $\eta$ is smooth and positive. Since $\rho|_{[-\frac{1}{2}, \frac{1}{2}]} \equiv 0$ and $\rho(t) = \rho(|t|)$, $\forall t \in (-1, 1)$, we have that $\eta|_{[-\frac{1}{2}, \frac{1}{2}]} \equiv 1$ and $\eta(x) = \eta(|x|)$, $\forall x \in (-1, 1)$. Therefore, the proof is complete.
\end{proof}

\begin{remark} \label{rmk-diffeo-interval-R}
    Let $\sigma$ be a diffeomorphism as in the above lemma. Since $\sigma(t) = t$, $\forall t \in [-\frac{1}{2}, \frac{1}{2}]$, we have that the inverse map $\sigma^{-1} : \mathbb{R} \to (-1, 1)$ has the form $\sigma^{-1}(s) = \tau(s)s$, $\forall s \in \mathbb{R}$, where $\tau : \mathbb{R} \to \mathbb{R}$ is a smooth positive function such that $\tau|_{[-\frac{1}{2}, \frac{1}{2}]} \equiv 1$, and $\tau(s) = \tau(|s|)$, $\forall s \in \mathbb{R}$.
\end{remark}

\begin{proposition} \label{prop-extension-vector-bundle}
    Let $\pi : E \to M$ be a smooth vector bundle on a smooth manifold $M$. Let $g$ be a smooth bundle metric on $E$, and $\delta : M \to \R$ a positive smooth function on $M$. Let $W = \big\{(p, v) \in E \big| |v|_g < \delta(p)\big\}$, and $W' = \big\{(p, v) \in E \big| |v|_g < \frac{1}{2}\delta(p)\big\}$. There exists a diffeomorphism
    $$\Phi : W \to E$$
    such that $\Phi|_{W'} = {\rm Id}_{W'}$.
\end{proposition}

\begin{proof}
    Let $\sigma : (-1, 1) \to \mathbb{R}$ be a diffeomorphism as in Lemma \ref{lem-diffeo-interval-R} with the form $\sigma(t) = \eta(t)t$. Then we define $\Phi : W \to E$ by
    $$\Phi(p, v) = \big(p, \eta(\frac{|v|_g}{\delta(p)})v\big), \quad \forall (p, v) \in W.$$
    It is clear that $\Phi$ is bijective. Since $\eta|_{[-\frac{1}{2}, \frac{1}{2}]} \equiv 1$, we have that $\Phi$ is smooth and $\Phi|_{W'}$ is the identity map of $W'$. Moreover, by Remark \ref{rmk-diffeo-interval-R}, the inverse map $\Phi^{-1}$ has a similar expression as $\Phi$, and then the smoothness of $\Phi^{-1}$ is immediate. Therefore, the proposition has been proved.
\end{proof}

\begin{corollary} \label{cor-extension-vector-bundle}
    Let $\pi : E \to M$ be a smooth vector bundle over a smooth manifold $M$. Let $V \subset E$ be an arbitrary open subset of $E$ containing the zero-section $M$, and $F : V \to M'$ a smooth map from $V$ to a smooth manifold $M'$. Then there exists an open subset $W \subset V$ of $V$ containing $M$, and a smooth map $\tilde{F} : E \to M'$ such that $\tilde{F}|_W = F|_W$. Moreover, if $F$ is an embedding, $\tilde{F}$ can also be an embedding.
\end{corollary}
\begin{proof}
    Let $\delta : M \to \R$ be a smooth positive function such that
    $$\big\{(p, v) \in E \big| |v|_g < \delta(p)\big\} \subset V.$$
    Then applying Proposition \ref{prop-extension-vector-bundle} to the open subset $U = \big\{(p, v) \in E \big| |v|_g < \delta(p)\big\}$, the corollary is immediately proved.
\end{proof}

\bigskip
Haoyuan Gao (hy\_gao@fudan.edu.cn)

Shanghai Center for Mathematical Sciences, Fudan University, Jiangwan Campus, Shanghai 200438, P. R. China

%\bigskip

\end{document}